\documentclass{amsart}

\usepackage{amssymb,amsmath,amsfonts, amsthm}
\usepackage{amscd,enumerate}
\newtheorem{theorem}{Theorem}[section]
\newtheorem{proposition}[theorem]{Proposition}
\newtheorem{lemma}[theorem]{Lemma}
\newtheorem{conjecture}[theorem]{Conjecture}
\newtheorem{problem}[theorem]{Problem}

\theoremstyle{remark}
\newtheorem{remark}[theorem]{Remark}

\theoremstyle{definition}

\newtheorem{convention}[theorem]{Convention}

\newcommand{\R}{\Bbb R}
\newcommand{\G}{G(L)}
\newcommand{\C}{\Bbb C}
\newcommand{\Z}{\Bbb Z}
\newcommand{\D}{\Delta}

\newcommand{\tr}{{\mathrm{tr}\,}}
\newcommand{\la}{\langle}
\newcommand{\ra}{\rangle}

\numberwithin{equation}{section}

\newcommand{\disp}{\displaystyle}

\begin{document}

\title[Twisted Alexander polynomials of hyperbolic links]
{Twisted Alexander polynomials of hyperbolic links}

\author{Takayuki Morifuji and Anh T. Tran}

\begin{abstract}
In this paper 
we apply the twisted Alexander polynomial to study 
the fibering and genus detecting problems for oriented links. 
In particular 
we generalize a conjecture of Dunfield, Friedl and Jackson 
on the  torsion polynomial of hyperbolic knots
to hyperbolic links, and confirm it for an infinite family of hyperbolic 
$2$-bridge links. Moreover we consider a similar problem for 
parabolic representations of $2$-bridge link groups. 
\end{abstract}



\address{Department of Mathematics,
Hiyoshi Campus, Keio University, 
Yokohama 223-8521, Japan}
\email{morifuji@z8.keio.jp}

\address{Department of Mathematical Sciences, The University of Texas at Dallas, 
Richardson, TX 75080, USA}
\email{att140830@utdallas.edu}

\maketitle

\section{Introduction}\label{sec:intro}

The twisted Alexander polynomial, a generalization of the classical Alexander polynomial 
\cite{Alexander28-1}, is defined for a pair consisting of  a group and 
its representation into a linear group. This invariant was first introduced by Lin \cite{Lin01-1} for knots 
in the $3$-sphere $S^3$ and by Wada \cite{Wada94-1} for 
finitely presentable groups which include the link groups. 
In recent years 
a theory of twisted Alexander polynomials has been rapidly developed and 
contributed to solving various important problems in low-dimensional topology, 
especially, in the theory of knots and links. 
However 
there seems to be still remained to study this invariant. 
As for recent developments on twisted Alexander polynomials and their applications, 
we refer to the survey papers \cite{FV10-1}, \cite{Morifuji15-1} 
and the references therein. 

In \cite{DFJ12-1}, based on huge experimental calculations, 
Dunfield, Friedl and Jackson conjectured that for a hyperbolic knot $K$ in $S^3$, 
i.e. the complement $S^3\setminus K$ has a complete hyperbolic structure of finite volume, 
the twisted Alexander polynomial associated to a lift of the holonomy representation detects the genus and 
fiberedness of $K$. At the present time, 
the conjecture has been confirmed for all hyperbolic knots with at most $15$ crossings \cite{DFJ12-1}, 
a certain infinite family of hyperbolic $2$-bridge knots (see \cite{Morifuji12-1}, \cite{MT}). 
Moreover, 
Agol and Dunfield showed in \cite{AD15-1} that the twisted Alexander polynomial detects the genus 
for a large class of hyperbolic knots, which includes many knots whose ordinary Alexander polynomial is trivial. 

The purpose of this paper 
is to apply the twisted Alexander polynomial to study 
the fibering and genus detecting problems for oriented links. 
In particular, 
we will generalize the conjecture of Dunfield, Friedl and Jackson to a hyperbolic link $L$ in $S^3$. 
In fact, we conjecture that the twisted Alexander polynomial associated 
to a lift of the holonomy representation, 
say $\rho_0:\pi_1(S^3\setminus L)\to SL(2,\C)$, 
detects the Thurston norm \cite{Thurston86-1} and fiberedness of 
an \textit{oriented}\, hyperbolic link, 
and show that the conjecture holds true for an infinite family of  hyperbolic $2$-bridge links. 
Actually we will show in Section~\ref{sec:examples} the following theorem. 

\begin{theorem}[Theorem~\ref{hyp_thm}]\label{thm:1.1}
For the double twist link $L$ as in Figure~2, 
the twisted Alexander polynomial $\D_{L,\rho_0}(t)$ associated to $\rho_0$ 
determines the Thurston norm. Moreover $L$ is fibered if and only if $\D_{L,\rho_0}(t)$ 
is monic. 
\end{theorem}

As is well-known, 
these topological properties of a $2$-bridge link are detected by 
the reduced Alexander polynomial (see \cite{Crowell59-1}, \cite{Murasugi58-1}, \cite{Murasugi63-1}). 
However there seems to be no a priori reason 
that the same must be true for the twisted Alexander polynomial. 

Since 
a lift of the holonomy representation of a hyperbolic link $L$ is one of the \textit{parabolic}\, representations 
(namely, it is a nonabelian $SL(2,\C)$-representation and the traces of the images of all the meridians of $L$ are two), 
it is natural to consider the following problem: 
\textit{For an oriented hyperbolic link $L$ and its parabolic representation 
$\rho:\pi_1(S^3\setminus L)\to SL(2,\C)$, does the twisted Alexander polynomial associated to 
$\rho$ determine the Thurston norm and fiberedness of $L$?}\, 
In this paper, 
we give a partial answer to this question in the case of a $2$-bridge link. 
More precisely we show that not all parabolic representations detect the genus 
(in this case, the Thurston norm is equivalent to the genus) of a hyperbolic $2$-bridge link. 

This paper is organized as follows. In Section~\ref{sec:preliminaries}, 
we briefly review some basic materials for the $SL(2,\C)$-character variety of a finitely generated group 
and the twisted Alexander polynomial of an oriented link associated to a two-dimensional linear representation. 
In Section~\ref{sec:hyperbolic torsion polynomials} 
we review a conjecture of Dunfield, Friedl and Jackson for hyperbolic knots
and state its generalization for oriented hyperbolic links. 
Section~\ref{sec:examples} is devoted to the calculation of the loci of the character variety which 
characterize fiberedness and the genus of a wide family of $2$-bridge links. 
This result can be regarded as a generalization of \cite{KM12-1} and \cite{KKM13-1} 
which discussed the same problem in the case of knots. 
In the last section, we give an answer to the question 
on parabolic representations mentioned above. 

\section*{Acknowledgements} 
The authors would like to thank Nathan Dunfield, Hiroshi Goda and Takahiro Kitayama for useful comments. 
The first author has been partially supported by JSPS KAKENHI Grant Number 26400096. 
The second author has been partially supported by a grant from the Simons Foundation (\#354595 to Anh Tran).

\section{Preliminaries}\label{sec:preliminaries}

In this section we give several standard definitions, 
and put a convention on the links which we will handle throughout the paper. 

\subsection{Oriented links}\label{subsec:alternating-link}

A $\mu$-component \textit{link}\, $L$ is the union of $\mu$ ordered, oriented 
and pairwise disjoint circles $L_i$ embedded in the $3$-sphere $S^3$. 
Two links $L$ and $L'$ are \textit{equivalent}\, 
if and only if 
$\mu=\mu'$ and there exists an orientation preserving homeomorphism 
$f$ of $S^3$ to itself such that $f(L_i)=L_i'$ and $f|_{L_i}$ is also 
orientation preserving for any $i$. 
A \textit{knot}\, is nothing but a $1$-component link. 

A \textit{Seifert surface} of an oriented link $L \subset S^3$ is a 
connected, oriented, compact surface $S$ embedded in $S^3$ whose boundary is $L$ such that 
the orientation on $L$ is just the induced orientation from $S$. 
The link $L$ is called \textit{fibered}\, 
if the exterior $E_L= S^3 \setminus \text{int}(N(L))$ has a structure of a surface bundle over the circle 
such that a Seifert surface of $L$ represents a fiber.
The \textit{genus}\, $g=g(L)$ of $L$ is the 
minimum of the genera of all its Seifert surfaces. 
We note that fiberedness and the genus of a link depend on a choice of 
orientations of the link. 

A link $L$ in $S^3$ is called \textit{hyperbolic}\, if the exterior 
$E_L $ has a complete hyperbolic structure of finite volume. 

Let $N$ be a compact connected orientable $3$-manifold and $\sigma\in H^1(N;\Z)$. 
The \textit{Thurston norm}\, of $\sigma$ is defined as 
$$
||\sigma||_T
=\min\{\chi_-(S)\,|\,S\subset N~ \text{properly embedded surface dual to}~\sigma\},
$$
where for a given surface $S$ with connected components $S_1\cup\cdots\cup S_k$, 
we define $\chi_-(S)=\sum_{i=1}^k\max\{-\chi(S_i),0\}$. 
Thurston showed in \cite{Thurston86-1} that this defines a seminorm on $H^1(N;\Z)$ 
and moreover on $H^1(N;\R)$. 

\begin{convention}\label{non-split}
In this paper we always consider \textit{oriented}\, links. 
In addition, we assume throughout this paper that links are \textit{non-split}. 
\end{convention}

\subsection{Character varieties}\label{subsec:character-variety}

Let $G$ be a finitely generated group.
The \textit{variety of representations} $R(G)$
of $G$ is the set of
$SL(2,\C)$-representations, i.e.
$R(G)=\mathrm{Hom}(G,SL(2,\C))$.
Since $G$ is finitely generated, $R(G)$ can be embedded in
a product $SL(2,\C)\times\cdots\times SL(2,\C)$ by mapping
each representation to the image of a generating set.
In this manner,
$R(G)$ is an affine algebraic set whose defining polynomials
are induced by the relators of a presentation of $G$.
It is known that 
this structure is independent of the choice of
presentations of $G$ up to isomorphism.

A representation $\rho:G\to SL(2,\C)$ is said to be
\textit{abelian}\/ if $\rho(G)$ is an abelian subgroup of
$SL(2,\C)$. A representation $\rho$ is called \textit{reducible}\/
if there exists a proper invariant subspace in $\C^2$
under the action of $\rho(G)$.
This is equivalent to saying that $\rho$ can be conjugated to a
representation whose image consists of upper triangular matrices. 
When $\rho$ is not reducible, it is called
\textit{irreducible}.

Given a representation $\rho\in R(G)$, its \textit{character}\/ is
the map $\chi_\rho:G\to \C$ defined by
$\chi_\rho(\gamma)=\tr(\rho(\gamma))$ for $\gamma\in G$.
We denote the set of all characters by $X(G)$.
For a given element $\gamma\in G$,
we define the map $\tau_\gamma:X(G)\to\C$ by
$\tau_\gamma(\chi)=\chi(\gamma)$.
Then
it is known that $X(G)$ is an affine algebraic set
which embeds in $\C^N$ with coordinates
$(\tau_{\gamma_1},\ldots,\tau_{\gamma_N})$
for some $\gamma_1,\ldots,\gamma_N\in G$.
This affine algebraic set is called the \textit{character variety}\/
of $G$.
We note that the set $\{\gamma_1,\ldots,\gamma_N\}$ can
be chosen to contain a generating set of $G$.
The projection $\mathfrak{t}:R(G)\to X(G)$ given by
$\mathfrak{t}(\rho)= \chi_\rho$ is surjective.

For a \textit{link group}\, $G(L)=\pi_1 (E_L)$, 
namely the fundamental group of the exterior $E_L$ of $L$ in $S^3$, 
we write $R(L)=R(G(L))$ and $X(L)=X(G(L))$ for simplicity.

\subsection{Twisted Alexander polynomials}\label{subsec:TAP}

Let 
$L=L_1\sqcup\cdots\sqcup L_\mu$ be a $\mu$-component oriented link in $S^3$. 
We choose and fix a Wirtinger presentation of $G(L)$. 
That is, 
given a regular projection of the link $L$, 
we assign to each overpass a generator $x_i$, 
and to each crossing a relator 
$x_ix_jx_k^{-1}x_j^{-1}$, as in Figure 1 (the orientation of the under crossing arc is irrelevant). 
Thus we obtain a presentation of $G(L)$ with $q$ generators and $q$ relators, 
$\la
x_1,\ldots,x_q\,|\,r_1,\ldots,r_q
\ra$. 
After some reordering of the indices, the relators $r_1,\ldots,r_q$ satisfy 
$\prod_{i=1}^q r_i^{\pm1}=1$. 
This means that any one of the relators is a consequence of the other $q-1$ 
relators. We remove one of the relators and call the resulting presentation 
$$
G(L)
=
\la
x_1,\ldots,x_q\,|\,r_1,\ldots,r_{q-1}
\ra
$$
a \textit{Wirtinger presentation}\, of $G(L)$. 

\begin{figure}\label{fig:1}
\setlength{\unitlength}{0.1mm}
\thicklines{
\begin{picture}(320,250)(-85,-20)
\put(90,-30){\line(0,1){250}}
\put(77.5,210){$\blacktriangle$}
\put(-55,95){\line(1,0){120}}
\put(110,95){\line(1,0){120}}
\put(100,-30){\large{$x_j$}}
\put(-60,110){\large{$x_k$}}
\put(220,110){\large{$x_i$}}
\end{picture}
}
\caption{A relator $x_ix_jx_{k}^{-1}x_{j}^{-1}$.}
\end{figure}
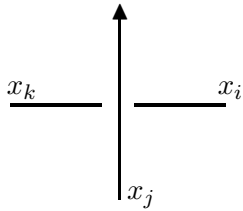

The abelianization homomorphism 
$$
\alpha_L:G(L)\to H_1(E_L;\Z)\cong\Z^{\oplus\mu}
=
\la t_1\ra\oplus\cdots\oplus\la t_\mu \ra
$$
is given by assigning to each generator $x_i$ the meridian element 
$t_k\in H_1(E_L;\Z)$ of the corresponding component $L_k$ of $L$. 
Here we denote the sum in $\Z$ multiplicatively. 
Moreover we consider the surjective homomorphism 
$p:H_1(E_L;\Z)\to \Z=\la t\ra$ defined by $t_i\mapsto t$ 
and for simplicity 
denote the composition map $p\circ\alpha_L:G(L)\to\Z$ by $\alpha$. 

In this paper 
we consider a representation of $G(L)$ into the two-dimensional special linear group 
$SL(2,\C)$, say $\rho:G(L)\to SL(2,\C)$. 
The maps $\rho$ and $\alpha$ naturally induce two ring
homomorphisms $\tilde{\rho}: {\Z}[G(L)] \rightarrow M(2,{\C})$ and
$\tilde{\alpha}:{\Z}[G(L)]\rightarrow {\Z}[t^{\pm1}]$, where
${\Z}[G(L)]$ is the group ring of $G(L)$ and $M(2,{\C})$ is the
matrix algebra of degree $2$ over ${\C}$. Then
$\tilde{\rho}\otimes\tilde{\alpha}$ defines a ring homomorphism
${\Z}[G(L)]\to M\left(2,{\C}[t^{\pm1}]\right)$. 
Let $F_q$ denote the
free group on generators $x_1,\ldots,x_q$ and
$$
\Phi:{\Z}[F_q]\to M\left(2,{\C}[t^{\pm1}]\right)
$$
the composition of the surjection
$\tilde{\phi}:{\Z}[F_q]\to{\Z}[G(L)]$
induced by the presentation of $G(L)$
and the map
$\tilde{\rho}\otimes\tilde{\alpha}:{\Z}[G(L)]\to M(2,{\C}[t^{\pm1}])$.

Let us consider the $(q-1)\times q$ matrix $M$
whose $(i,j)$-entry is the $2\times 2$ matrix
$$
\Phi\left(\frac{\partial r_i}{\partial x_j}\right)
\in M\left(2,{\C}[t^{\pm1}]\right),
$$
where
$\disp{\frac{\partial}{\partial x}}$
denotes the free differential. 
For
$1\leq j\leq q$,
let us denote by $M_j$
the $(q-1)\times(q-1)$ matrix obtained from $M$
by removing the $j$th column.
We regard $M_j$ as
a $2(q-1)\times 2(q-1)$ matrix with coefficients in
${\C}[t^{\pm1}]$. 
Then Wada's \textit{twisted Alexander polynomial}\/
\cite{Wada94-1} of a link $L$
associated with a representation $\rho:\G\to SL(2,{\C})$ is defined
to be the rational function
$$
\D_{L,\rho}(t)
=\frac{\det M_j}{\det\Phi(1-x_j)}
$$
and well-defined up to multiplication by
$t^{2k}~(k\in{\Z})$. 

\begin{remark}\label{rmk:trivial-rep}
By definition, 
$\D_{L,\rho}(t)$ is a rational function in the variable $t$, 
but it will be a Laurent polynomial if $L$ is a link with two or more components~\cite[Proposition~9]{Wada94-1}, or 
$L$ is a knot $K$ and $\rho$ is non-abelian~\cite[Theorem~3.1]{KtM05-1}. 
\end{remark}

We note that if $\rho$ and $\rho'$ are mutually conjugate
$SL(2,\C)$-representations,
then $\D_{L,\rho}(t) = \D_{L,\rho'}(t)$ holds.
If $\rho$ and $\rho'\colon \G \to SL(2,\C)$ are irreducible
representations with $\chi_\rho = \chi_{\rho'}$, then $\rho$ is conjugate to $\rho'$
(see \cite[Proposition 1.5.2]{CS83-1}), and hence
$\D_{L,\rho}(t) = \D_{L,\rho'}(t)$.
Moreover if $\rho$ and $\rho'$ are reducible representations  with $\chi_\rho = \chi_{\rho'}$, 
then $\D_{L,\rho}(t)$ and $\D_{L,\rho'}(t)$ 
are determined by diagonal
entries of images of $\rho$ and $\rho'$ and hence 
they are equivalent. 
Therefore, we can define the twisted Alexander polynomial
associated with $\chi \in X(L)$ to be $\D_{L,\rho}(t)$
where $\chi = \chi_\rho$, 
and we denote it by $\D_{L,\chi}(t)$. 

It is known that we can write the twisted Alexander polynomial 
$\D_{L,\chi}(t)$ without any ambiguity as
\[
\D_{L,\chi}(t)
=
\sum_{j=0}^{2l}\psi_j(\chi) t^j
\]
with $\C$-valued functions $\psi_j$ on $X(L)$ such that 
$\psi_k=\psi_{2l-k}~(0\leq k\leq l)$ (see \cite[Theorem 1.5]{FKK12-1}), 
where $l=||\alpha||_T$ is the Thurston norm of $\alpha$ (see \cite[Theorem~1.1]{FK06-1}).  

For a subvariety $X_0$ of $X(L)$, we say that $\psi_n$ is
\emph{the coefficient of the highest degree term of $\D_{L,\chi}(t)$ on $X_0$}
if $\psi_m\equiv 0$ for $m>n$ and $\psi_n\not\equiv 0$ on $X_0$.
Moreover we say $\D_{L,\rho}(t)$ 
(respectively, $\D_{L,\chi}(t)$) 
is \textit{monic}\/ 
if the coefficient of the highest degree term 
of $\D_{L,\rho}(t)$ 
(respectively, $\D_{L,\chi}(t)$) is one.  This makes sense because 
the twisted Alexander polynomial is well-defined up to multiplication by 
$t^{2k}~(k\in\Z)$. 
It is known that the twisted Alexander polynomial of a fibered link is monic 
for every non-abelian representation \cite[Theorem~1.1]{FK06-1} (see \cite{Cha03-1}, \cite{GKM05-1} 
for the case of fibered knots).

\section{Hyperbolic torsion polynomials}\label{sec:hyperbolic torsion polynomials}

\subsection{A conjecture of Dunfield, Friedl and Jackson}\label{subsec:DFJ}
Let $N$ be a hyperbolic $3$-manifold of finite volume. 
Then there is a faithful representation 
$\overline{\rho}_0:\pi_1(N)\to PSL(2,\C)\cong \mathrm{Isom}^+(\mathbb{H}^3)$, 
where $\mathbb{H}^3$ denotes the upper half-space model of the hyperbolic $3$-space, 
with discrete image such that $\mathbb{H}^3/\mathrm{Im}\,\overline{\rho}_0\cong N$. 
The representation $\overline{\rho}_0$ is called the \textit{holonomy representation}\, 
and is unique up to conjugation. 
It is known that a peripheral torus subgroup of $\overline{\rho}_0(\pi_1(N))\subset PSL(2,\C)$ 
is conjugate to a group of cosets of matrices of the form 
$\begin{pmatrix}1&\nu\\0&1\end{pmatrix}$, where $\nu\in\C$. 
In particular, the traces (defined up to sign) of the elements of such a subgroup are $\pm2$. 
Using a result of Thurston, 
$\overline{\rho}_0$ may be lifted to a representation in 
$SL(2,\C)$ which is also faithful and has discrete image. 

In \cite{DFJ12-1}, 
Dunfield, Friedl and Jackson studied the twisted Alexander polynomial 
$\D_{K,\rho_0}(t)$ for $N=E_K$, the exterior of a hyperbolic knot $K$ in $S^3$ 
and $\rho_0:G(K)\to SL(2,\C)$, a lift of $\overline{\rho}_0$ such that 
$\tr \rho_0(\mu_K)=2$ for a meridian $\mu_K$. 
They call $\D_{K,\rho_0}(t)$ the \textit{hyperbolic torsion polynomial}\, and denote it 
by $\mathcal{T}_K(t)$. Moreover, 
based on huge numerical calculations, they conjectured the following. 

\begin{conjecture}[{Dunfield-Friedl-Jackson~\cite[Conjecture~1.4]{DFJ12-1}}]\label{conj:DFJ}
For a hyperbolic knot $K$ in $S^3$, 
the hyperbolic torsion polynomial $\mathcal{T}_K(t)$ determines the Thurston norm $x(K)$ 
or equivalently the genus of $K$. 
Moreover, 
the knot $K$ is fibered if and only if $\mathcal{T}_K(t)$ is monic. 
\end{conjecture}

\begin{remark}\label{rmk:norm_genus}
Let $\alpha_K\in H^1(E_K;\Z)=\mathrm{Hom}(G(K),\Z)$ 
be the abelianization, then we can see $||\alpha_K||_T=2g(K)-1$ and denote it by $x(K)$. 
\end{remark}

Conjecture~\ref{conj:DFJ} has been verified for all hyperbolic knots with at most 15 crossings \cite{DFJ12-1}, 
hyperbolic twist knots \cite{Morifuji12-1} and a certain wide class of hyperbolic $2$-bridge knots \cite{MT}. 
Recently Agol and Dunfield showed the former part of the conjecture for a large class of 
hyperbolic knots in $S^3$ which includes all special arborescent knots and many 
knots whose ordinary Alexander polynomial is trivial (see \cite{AD15-1} for details). 

\subsection{A generalization}\label{subsection:DFJ for links}

For a $\mu$-component hyperbolic link $L$ there are $2^\mu$ possible lifts of the 
holonomy representation $\overline{\rho}_0:G(L)\to PSL(2,\C)$ to 
an $SL(2,\C)$-representation. 
It is known that there is a canonical one-to-one correspondence between 
the set of lifts of the holonomy representation and the set of spin 
structures of the exterior of a link (see \cite{MP14-1}). 
Among them we focus on the lift 
$\rho_0:G(L)\to SL(2,\C)$ such that the images of the meridians of $L$ by $\rho_0$ are 
matrices in $SL(2,\C)$ with the trace two. 
Similar to the case of knots, 
in this paper, 
we call this kind of non-abelian $SL(2,\C)$-representation (or character) 
of the link group parabolic. 
That is, a non-abelian $SL(2,\C)$-representation $\rho:G(L)\to SL(2,\C)$ (or character $\chi_\rho$) is 
called \textit{parabolic}\, if the images of the meridians of $L$ by $\rho$ are matrices with the trace two. 
Then we propose the following conjecture. 

\begin{conjecture} \label{conj:hyp}
For a $\mu$-component oriented hyperbolic link $L$ in $S^3$ 
the twisted Alexander polynomial $\D_{L,\rho_0}(t)$ determines the Thurston norm, 
namely, 
we have
$$
\deg \Delta_{L, \rho_0}(t) = 2 ||\alpha||_T
$$
for $\alpha\in H^1(E_L;\Z)$ given by sending each meridian to one. 
Moreover, 
the link $L$ is fibered if and only if 
$\Delta_{L, \rho_0}(t)$ is monic.
\end{conjecture}

\begin{remark}\label{rmk:Alexander-Thurston}
For an alternating link $L$, 
it is known that the Thurston norm $||\alpha||_T$ 
coincides with the Alexander norm $||\alpha||_A$ for $\alpha\in H^1(E_L;\Z)$ 
as above (see \cite{McMullen02-1} for details). The Alexander norm is determined by 
the (multi-variable) Alexander polynomial of $L$, so that we have 
$||\alpha||_A=\deg\D_{L}(t)-1$, where $\D_L(t)$ denotes the reduced Alexander polynomial. 
Hence the equality in Conjecture~\ref{conj:hyp} will be 
$\deg \Delta_{L, \rho_0}(t) = 4g(L)+2(\mu-2)$ for a $\mu$-component alternating 
hyperbolic link $L$. 
\end{remark}

In the next section, 
we show that Conjecture~\ref{conj:hyp} holds true for a wide family of $2$-bridge links 
which are called the double twist links. 

\section{Fibering and genus detecting problems}\label{sec:examples}

In this section
we discuss the fibering and genus detecting problems for the double twist links. 
In particular, for  these links, 
we specify a finite number of loci in the character variety which characterize fiberedness and the genus 
of the links. This result can be regarded as a generalization of 
\cite[Theorems~4.3 and 4.4]{KM12-1}. As a consequence, 
we will show in Subsection~\ref{subsec:main theorem} that Conjecture~\ref{conj:hyp} holds true 
for the double twist links. 

We start by reviewing Chebyshev polynomials and their properties.

\subsection{Chebyshev polynomials (1)} 
Let $S_k(q)$ be the Chebyshev polynomials of the second kind defined by $S_0(q)=1$, $S_1(q)=q$ and $S_{k+1}(q) =  q \, S_k(q) - S_{k-1}(q)$ for all integers $k$. Similarly, let $T_k(q)$ be the Chebyshev polynomials of the first kind defined by $T_0(q)=2$, $T_1(q)=q$ and $T_{k+1}(q) =  q \, T_k(q) - T_{k-1}(q)$ for all integers $k$. Note that $T_k(q) = S_k(q) - S_{k-2}(q)$. 

The following three lemmas are elementary, and hence their proofs are omitted.

\begin{lemma} \label{chev1}
Write $q = v + v^{-1}$. Then $$T_k(q) = v^k + v^{-k}.$$ 
We have $S_k(2) = k+1$ and $S_k(-2) = (-1)^k(k+1)$. If $v \not= \pm 1$ then $$S_k(q) = \frac{v^{k+1} - v^{-(k+1)}}{v -  v^{-1}}.$$ 
In particular, if $q = 2\cos \beta$, where $\disp{\frac{\beta}{\pi} \in \R \setminus \Z}$, 
then $\disp{S_k(q) = \frac{\sin (k+1)\beta}{\sin \beta}}$.
\end{lemma}

\begin{lemma} \label{chev0}
We have
$$S^2_k(q) + S^2_{k-1}(q) - q S_k(q) S_{k-1}(q) =1.$$
\end{lemma}

\begin{lemma} \label{chev2}
For any positive integer $k$ we have 
$$S_k(q) = \prod_{j=1}^k \left(q - 2 \cos \frac{j \pi}{k+1}\right).$$
In particular, all the roots of $S_k(q)$ are real numbers strictly between $-2$ and $2$. 
\end{lemma}

\begin{lemma} \label{chev3}
\begin{enumerate}
\item  
If $k \ge 1$ is even, then $$T_k(q) -2 = (q-2)(q+2) \prod_{j=1}^{\frac{k}{2}-1} \left( q - 2 \cos \frac{2j\pi}{k} \right)^2.$$
\item  
If $k \ge 1$ is odd, then $$T_k(q) -2 = (q-2) \prod_{j=1}^{\frac{k-1}{2}} \left( q - 2 \cos \frac{2j\pi}{k} \right)^2.$$
\end{enumerate}
\end{lemma}

\begin{proof}
Write $q = v + v^{-1}$. By Lemma \ref{chev1} we have $T_k(q) = v^k + v^{-k}.$ Then 
\begin{eqnarray*}
T_k(q) -2 &=& v^{-k} (v^k-1)^2  \\
          &=& v^{-k} (v-1)^2 \prod_{1 \le j \le k-1}  \left( v - e^{2\pi i\frac{j}{k}} \right)^2.
\end{eqnarray*}
Grouping the terms involving $j$ and $k-j$ we have 
\begin{eqnarray*}
\left( v - e^{2\pi i\frac{j}{k}} \right) \left( v - e^{2\pi i\frac{k-j}{k}} \right) &=& \left( v - e^{2\pi i\frac{j}{k}} \right) \left( v - e^{-2\pi i\frac{j}{k}} \right) \\
&=& v^2+1 - 2v \cos \frac{2j\pi}{k}.
\end{eqnarray*}

Suppose $k$ is even. Since $e^{2\pi i\frac{j}{k}} = -1 $ if $\disp{j=\frac{k}{2}}$, we have
\begin{eqnarray*}
T_k(q) -2 &=&  v^{-k} (v-1)^2 (v+1)^2 \prod_{1 \le j \le \frac{k}{2}-1}  \left( v^2+1 - 2v \cos \frac{2j\pi}{k} \right)^2 \\
&=& (q-2)(q+2) \prod_{j=1}^{\frac{k}{2}-1} \left( q - 2 \cos \frac{2j\pi}{k} \right)^2.
\end{eqnarray*}
The case of odd $k$ is similar.
\end{proof}

\begin{lemma} \label{chev4}
\begin{enumerate}
\item 
If $k \ge 1$ is odd, then $$T_k(q) - q = (q-2)(q+2) \prod_{j=1}^{\frac{k-3}{2}} \left( q - 2 \cos \frac{2j\pi}{k-1} \right) \prod_{j=1}^{\frac{k-1}{2}} \left( q - 2 \cos \frac{2j\pi}{k+1} \right).$$

\item 
If $k \ge 1$ is even, then $$T_k(q) - q = (q-2) \prod_{j=1}^{\frac{k-2}{2}} \left( q - 2 \cos \frac{2j\pi}{k-1} \right) \prod_{j=1}^{\frac{k}{2}} \left( q - 2 \cos \frac{2j\pi}{k+1} \right).$$
\end{enumerate}
\end{lemma}

\begin{proof}
The proof of Lemma \ref{chev4} is similar to that of Lemma \ref{chev3}.
\end{proof}

The following lemma is known, see e.g. \cite[Proposition 2.4]{Tr-Rtorsion}.

\begin{lemma}\label{chev5}
Suppose $Q = \begin{pmatrix}
Q_{11} & Q_{12}\\
Q_{21} & Q_{22}
\end{pmatrix} \in SL(2,\C)$. Let $q = \tr Q$. Then for any positive integer $k$ we have
$$Q^k = \begin{pmatrix}
S_k(q) - Q_{22} S_{k-1}(q) & Q_{12} S_{k-1}(q)\\
Q_{21} S_{k-1}(q) & S_k(q) - Q_{11} S_{k-1}(q)
\end{pmatrix}
$$ and
$$\det (I + Q + \cdots + Q^{k-1}) = \frac{T_k(q)-2}{q-2}.$$
\end{lemma}

\subsection{Double twist links}

For integers $m$ and $n$, consider the double twist link $L=J(2m+1,2n+1)$ 
which is the $2$-bridge link corresponding to the rational number 
$\disp{\frac{2n+1}{4mn+2(m+n)}}$ (see Figure~2). By \cite[Lemma~3.2]{Petersen-Tran} 
(and \cite[Section~2.3]{MPL} also), the link group of $L$  has a presentation 
$$G(L)= \left\la a, b \mid \big( (ab^{-1})^m ab (a^{-1}b)^m \big)^n (ab^{-1})^m 
= (b^{-1}a)^m \big( (ba^{-1})^m ba (b^{-1}a)^m \big)^n \right\ra,$$
where $a$ and $b$ are meridians of $L$ depicted in Figure 2. 

\begin{figure}\label{fig:2}
\setlength{\unitlength}{0.09mm}
\thicklines{
\begin{picture}(300,560)(380,30)
\put(216,300){$\blacktriangledown$}
\put(160,400){\line(1,0){50}}
\put(250,400){\line(1,0){50}}
\put(290,391){$\blacktriangleright$}
\put(290,425){$a$}
\put(760,400){\line(1,0){50}}
\put(850,400){\line(1,0){50}}
\put(740,390){$\blacktriangleleft$}
\put(740,420){$b$}
\put(230,0){\line(0,1){590}}
\put(230,0){\line(1,0){125}}
\put(330,100){\line(1,0){25}}
\put(330,100){\line(0,1){125}}
\put(355,0){\line(3,4){30}}
\put(397,60){\line(3,4){30}}
\put(430,0){\line(3,4){30}}
\put(478,60){\line(3,4){30}}
\put(353,99){\line(3,-4){74}}
\put(430,100){\line(3,-4){74}}
\put(630,0){\line(3,4){30}}
\put(675,60){\line(3,4){30}}
\put(630,100){\line(3,-4){75}}
\put(330,225){\line(1,0){150}}
\put(480,225){\line(0,1){40}}
\put(480,413){\line(3,2){100}}
\put(480,480){\line(3,2){100}}
\put(480,480){\line(3,-2){40}}
\put(540,440){\line(3,-2){40}}
\put(480,549){\line(3,-2){40}}
\put(541,507){\line(3,-2){40}}
\put(480,330){\line(3,-2){40}}
\put(540,290){\line(3,-2){40}}
\put(480,265){\line(3,2){100}}
\put(230,590){\line(1,0){250}}
\put(480,590){\line(0,-1){40}}
\put(580,548){\line(0,1){42}}
\put(580,590){\line(1,0){250}}
\put(830,590){\line(0,-1){590}}
\put(706,0){\line(1,0){125}}
\put(580,225){\line(1,0){150}}
\put(816,300){$\blacktriangle$}
\put(580,225){\line(0,1){40}}
\put(730,100){\line(0,1){125}}
\put(730,100){\line(-1,0){25}}
\put(530,44){{\tiny $\bullet~\bullet~\bullet$}}
\put(525,345){{\tiny $\bullet$}}
\put(525,365){{\tiny $\bullet$}}
\put(525,385){{\tiny $\bullet$}}
\put(410,115){{\small $2n+1$ crossings}}
\put(400,360){{\small $2m+1$}}
\end{picture}
}
\caption{The double twist link $J(2m+1,2n+1)$ with $m,n \ge 1$.}
\end{figure}
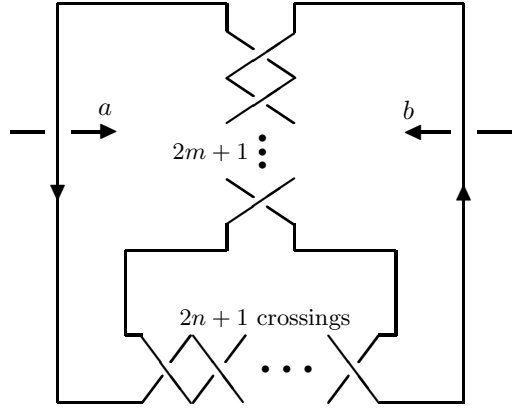

By setting $w=(b^{-1}a)^m \big( (ba^{-1})^m ba (b^{-1}a)^m \big)^n$, we can rewrite the link group as 
$$G(L) = \la a,b \mid  awa^{-1} = w\ra.$$

Let $\rho:G(L)\to SL(2,\C)$ be a non-abelian representation and 
$r=awa^{-1}w^{-1}$. We have 
$$\frac{\partial r}{\partial b} 
= a \left(\frac{\partial w}{\partial b} - wa^{-1}w^{-1} \frac{\partial w}{\partial b}\right) 
= a (1 - wa^{-1}w^{-1})\frac{\partial w}{\partial b}.$$
Hence 
$\disp{\Delta_{L,\rho} (t)= \det \Phi\left(\frac{\partial r}{\partial b}\right) 
\big/ \det \Phi(1-a) = \det \Phi\left(\frac{\partial w}{\partial b}\right)}$.

We consider only the case $m,n \ge 1$. The other cases are similar. Let $u = (ba^{-1})^m ba (b^{-1}a)^m$. Then $w = (b^{-1}a)^m u^n$. We have
$$\frac{\partial w}{\partial b} = \left( 1+ (b^{-1}a) + \cdots + (b^{-1}a)^{m-1} \right) (-b^{-1}) + (b^{-1}a)^m (1+u+ \cdots +u^{n-1})\frac{\partial u}{\partial b},$$
where
\begin{eqnarray*}
\frac{\partial u}{\partial b} &=& \frac{\partial  \big( b(a^{-1}b)^m(ab^{-1})^ma \big)}{\partial b} \\
&=& 1 + b (1+(a^{-1}b) + \cdots + (a^{-1}b)^{m-1})a^{-1} \\
&& - b(a^{-1}b)^m (1+(ab^{-1})+ \cdots + (ab^{-1})^{m-1})ab^{-1} .
\end{eqnarray*}

For $h \in G(L)$ we denote $\rho(h)$ by $H$. 
With the orientation of $L$ as in Figure 2, the genus is given by $g(L) = n$ and $L$ is fibered if and only if $m=1$. (These facts can be proved by computing the reduced Alexander polynomial of $L$ and then applying \cite[Theorem 1.1]{Murasugi63-1}.) Moreover, $\Phi(a) = t A$ and $\Phi(b) = t B$. Then $\Phi(u) = t^2 U$. 
The highest degree term (in $t$) of $\Delta_{L, \rho}(t)$ is 
\begin{eqnarray*}
&& \det \Phi\left( -(b^{-1}a)^m u^{n-1} b(a^{-1}b)^m (1+ab^{-1}+ \cdots + (ab^{-1})^{m-1})ab^{-1}\right) \\
&=&  \left(t^{2n-1})^2 \det \Phi(1+ab^{-1}+ \cdots + (ab^{-1})^{m-1}\right)\\
&=& t^{4n-2} \, \frac{T_m(z)-2}{z-2},
\end{eqnarray*}
where 
$z = \tr AB^{-1}$. Here we apply Lemma \ref{chev5} with $Q=AB^{-1}$. 
Similarly, the lowest degree term of $\Delta_{L, \rho}(t)$ is $\disp{t^{-2} \, \frac{T_m(z)-2}{z-2}}$. 

By \cite[Theorem 1.1]{Petersen-Tran} the non-abelian character variety of $L=J(2m+1,2n+1)$ is the zero set in $\C^3$ of the polynomial $$R(x,y,z) = S_{m-1}(z) S_n(v) - S_m(z) S_{n-1}(v),$$
where $x = \tr A, ~y=\tr B$ and 
\begin{eqnarray*}
v = \tr U &=& \big( x S_m(z) - y S_{m-1}(z) \big) \big( y S_m(z) - x S_{m-1}(z) \big) \\
&&-  z \big( S^2_m(z) + S^2_{m-1}(z) \big) + 4 S_m(z) S_{m-1}(z).
\end{eqnarray*}

\subsection{Genus} \label{subsec:4.1}
By Lemma \ref{chev3} we have $\disp{\frac{T_m(z)-2}{z-2}=0}$ if and only if 
$\disp{z = 2 \cos \frac{2j\pi}{m}}$ for some $\disp{1 \le j \le \frac{m}{2}}$. 

Suppose $\disp{z = 2 \cos \frac{2j\pi}{m}}$ for some $\disp{1 \le j < \frac{m}{2}}$. 
Then, by Lemma \ref{chev1}, $$S_{m-1}(z)=\frac{\sin m \frac{2j\pi}{m}}{\sin \frac{2j\pi}{m}}=0, \quad \text{and} \quad S_m(z)=\frac{\sin (m+1)\frac{2j\pi}{m}}{\sin \frac{2j\pi}{m}}=1.$$
Hence $R(x,y,z) = - S_{n-1}(v)$, where $v = xy - z$. Then, by Lemma \ref{chev2}, $R(x,y,z)=0$ 
if and only if $\disp{v = 2\cos \frac{k\pi}{n}}$ for some $1 \le k \le n-1$.

Suppose $z=-2$ (in this case $m$ must be even). Then $S_{m-1}(z)=(-1)^{m-1}m = -m$ and $S_m(z)=(-1)^m(m+1)=m+1$. Hence $$R(x,y,z) = - \left( mS_n(v) + (m+1) S_{n-1}(v) \right)$$ where $v = xy+2 + (m^2+m)(x+y)^2$. 

We have shown that the following. Let
$$Y_{j,k} = \left\{ x,y,z \in \C \Bigm| z = 2 \cos \frac{2j\pi}{m}, \, xy -z =  2\cos \frac{k\pi}{n} \right\}.$$
Let $v_1 = xy+2 + (m^2+m)(x+y)^2$ and
$$Y=\left\{x,y,z \in \C \bigm| z =-2, \, mS_n(v_1) + (m+1) S_{n-1}(v_1)=0 \right\}.$$

\begin{proposition}\label{pro:4.1}  
For the double twist link $L=J(2m+1,2n+1)$ with $m, n \ge 1$ and orientation as in Figure 2, 
\begin{enumerate}
\item
If $m$ is odd, then $\deg \Delta_{L, \chi}(t) = 4n = 4g(L)$ on 
$$\{R(x,y,z)=0\} \setminus \bigcup_{\substack{1 \le j < m/2, \\ 1 \le k \le n-1}} 
Y_{j,k}.$$
\item
If $m$ is even, then $\deg \Delta_{L, \chi}(t) = 4n = 4g(L)$ on 
$$\{R(x,y,z)=0\} \setminus Y \setminus \bigcup_{\substack{1 \le j < m/2, \\ 1 \le k \le n-1}} 
Y_{j,k}.$$
\end{enumerate}
\end{proposition}

\subsection{Fiberedness}\label{subsec:4.2} Here we consider non-fibered links $L$. Then we have $m > 1$.

Since $T_k(q)-q=T_k(q)-2-(q-2)$, 
by Lemma \ref{chev4} we have $\disp{\frac{T_m(z)-2}{z-2} = 1}$ if and only if 
$\disp{z = 2 \cos \frac{2j\pi}{m-1}}$ for some $\disp{1 \le j \le \frac{m-1}{2}}$, 
or $\disp{z = 2 \cos \frac{2k\pi}{m+1}}$ 
for some $\disp{1 \le k \le \frac{m+1}{2}}$.

Suppose $\disp{z = 2 \cos \frac{2j\pi}{m-1}}$ for some $\disp{1 \le j < \frac{m-1}{2}}$. 
Then, by Lemma \ref{chev1}, $$S_{m-1}(z)= \frac{\sin m \frac{2j\pi}{m-1}}{\sin \frac{2j\pi}{m-1}} = 1, \quad \text{and} \quad S_m(z)= \frac{\sin (m+1) \frac{2j\pi}{m-1}}{\sin \frac{2j\pi}{m-1}} = z.$$ 
Hence $R(x,y,z) = S_n(v)- z S_{n-1}(v)$, where $v = xy - (x^2+y^2-3)z+xyz^2-z^3$. 

Suppose $\disp{z = 2 \cos \frac{2j\pi}{m+1}}$ for some $\disp{1 \le j < \frac{m+1}{2}}$. 
Then, by Lemma \ref{chev1}, $$S_{m-1}(z)= \frac{\sin m \frac{2j\pi}{m+1}}{\sin \frac{2j\pi}{m+1}} = -1, \quad \text{and} \quad S_m(z)= \frac{\sin (m+1) \frac{2j\pi}{m+1}}{\sin \frac{2j\pi}{m+1}} = 0.$$ 
Hence $R(x,y,z) = - S_n(v)$, where $v = xy - z$. Then, by Lemma \ref{chev2}, $R(x,y,z)=0$ if and only if 
$\disp{v = 2\cos \frac{k\pi}{n+1}}$ for some $1 \le k \le n$.

Suppose $z=-2$ (in this case $m$ must be odd). Then $S_{m-1}(z)=(-1)^{m-1}m = m$ and $S_m(z)=(-1)^m(m+1)=-(m+1)$. Hence $$R(x,y,z) = mS_m(v) + (m+1) S_{n-1}(v) $$ where $v = xy+2 + (m^2+m)(x+y)^2$. 

We have shown that the following. Let
$$Z_{j,k} = \left\{ x,y,z \in \C \Bigm| z = 2 \cos \frac{2j\pi}{m+1}, \, xy -z =  2\cos \frac{k\pi}{n+1} \right\}.$$
Let $v_2 =  xy - (x^2+y^2-3)z+xyz^2-z^3$ and let
$$Z_l = \left\{x,y,z \in \C \Bigm| z=2 \cos \frac{2l\pi}{m-1}, \, S_n(v_2)- z S_{n-1}(v_2)=0 \right\}.$$

\begin{proposition}\label{pro:4.2}
For the (non-fibered) double twist link $L=J(2m+1,2n+1)$ with $m > 1$, $n \ge 1$ and orientation as in Figure 2, 
\begin{enumerate}
\item
If $m$ is even, then $\Delta_{L, \chi}(t)$ is non-monic on 
$$
\{R(x,y,z)=0\} \setminus  \bigcup_{\substack{1 \le j < (m+1)/2, \\ 1 \le k \le n}} 
Z_{j,k} \setminus \bigcup_{1 \le l < (m-1)/2} Z_l.
$$
\item
If $m$ is odd, then $\Delta_{L, \chi}(t)$ is non-monic on 
$$
\{R(x,y,z)=0\} \setminus Y \setminus  \bigcup_{\substack{1 \le j < (m+1)/2, \\ 1 \le k \le n}} 
Z_{j,k} \setminus \bigcup_{1 \le l < (m-1)/2} Z_l .
$$
\end{enumerate}
\end{proposition}

We now consider another orientation of $L$ which is obtained from the one in Figure 2 by changing the orientation of the component  corresponding to the meridian $b$. With this orientation we have $g(L) = m$ and $L$ is fibered if and only if $n=1$. Moreover $\Phi(a) = t A$ and $\Phi(b) = t^{-1} B$. Then $\Phi(u) =  U$. The highest degree term (in $t$) of $\Delta_{L, \rho}(t)$ is 
\begin{eqnarray*}
\det \Phi\left( (b^{-1}a)^m (1+u+ \cdots +u^{n-1})\right) &=& (t^{2m})^2 \det \Phi (1+u+ \cdots + u^{n-1}) \\
&=& t^{4m} \, \frac{T_n(v)-2}{v-2}.
\end{eqnarray*}
Similarly, the lowest degree term of $\Delta_{L, \rho}(t)$ is 
$\disp{t^{0} \, \frac{T_n(v)-2}{v-2}}$. 

\subsection{Genus}\label{subsec:4.3} 
Recall that \begin{eqnarray*}
v &=& \big( x S_m(z) - y S_{m-1}(z) \big) \big( y S_m(z) - x S_{m-1}(z) \big) \\
&&-  z \big( S^2_m(z) + S^2_{m-1}(z) \big) + 4 S_m(z) S_{m-1}(z).
\end{eqnarray*}

A similar argument as in Subsection~\ref{subsec:4.1} shows the following. 
Let
$$Y'_{j,k} = \left\{ x,y,z \in \C \Bigm| xy -z = 2 \cos \frac{2j\pi}{n}, \,  z =  2\cos \frac{k\pi}{m} \right\}$$
and
$$Y'=\left\{x,y,z \in \C \bigm| v =-2, \, nS_m(z) + (n+1) S_{m-1}(z)=0 \right\}.$$

\begin{proposition} \label{pro:4.3}
For the double twist link $L=J(2m+1,2n+1)$ with $m, n \ge 1$ and orientation obtained from the one in Figure 2 by changing the orientation of the component  corresponding to the meridian $b$, 
\begin{enumerate}
\item
 If $n$ is odd, then $\deg \Delta_{L, \chi}(t) = 4m = 4g(L)$ on 
$$\{R(x,y,z)=0\} \setminus \bigcup_{\substack{1 \le j < n/2, \\ 1 \le k \le m-1}} Y'_{j,k}.$$
\item
If $n$ is even, then $\deg \Delta_{L, \chi}(t) = 4m = 4g(L)$ on 
$$\{R(x,y,z)=0\} \setminus Y' \setminus \bigcup_{\substack{1 \le j < n/2, \\ 1 \le k \le m-1}} 
Y'_{j,k}.
$$
\end{enumerate}
\end{proposition}

\subsection{Fiberedness}\label{subsec:4.4} 

As for fiberedness of $J(2m+1,2n+1)$ with the same orientation as in Proposition~\ref{pro:4.3}, 
we have the following. 
Let
$$Z'_{j,k} = \left\{ x,y,z \in \C \Bigm| xy -z = 2 \cos \frac{2j\pi}{n+1}, \,  z =  2\cos \frac{k\pi}{m+1} \right\}$$
and
$$Z'_l = \left\{x,y,z \in \C \Bigm| v=2 \cos \frac{2l\pi}{n-1}, \, S_m(z)- v S_{m-1}(z)=0 \right\}.$$

\begin{proposition} \label{pro:4.4} 
For the non-fibered double twist link $L=J(2m+1,2n+1)$ with $m \ge 1$, $n > 1$ and orientation obtained from the one in Figure 2 by changing the orientation of the component  corresponding to the meridian $b$, 
\begin{enumerate}
\item
If $n$ is even, then $\Delta_{L, \chi}(t)$ is non-monic on 
$$
\{R(x,y,z)=0\} \setminus \bigcup_{\substack{1 \le j < (n+1)/2, \\ 1 \le k \le m}} 
Z'_{j,k} \setminus \bigcup_{1 \le l < (n-1)/2} 
Z'_l.
$$
\item
If $m$ is odd, then $\Delta_{L, \chi}(t)$ is non-monic on 
$$
\{R(x,y,z)=0\} \setminus Y' \setminus \bigcup_{\substack{1 \le j < (n+1)/2, \\ 1 \le k \le m}} 
Z'_{j,k} \setminus \bigcup_{1 \le l < (n-1)/2} 
Z'_l.
$$
\end{enumerate}
\end{proposition}

The remaining two orientations of $L$ are similar. They are just mirror images of the previous two. In general, there are $2^\mu$ ways to orient a $\mu$-component link. Moreover, up to taking mirror images, there are only $2^{\mu-1}$ orientations to consider.

\begin{remark}\label{rmk:similar}
Similar results to Propositions \ref{pro:4.1}--\ref{pro:4.4} can be stated for all double twist links $J(2m+1, 2n+1)$, where $m, n$ are integers. However, for simplicity, we choose to  present only the case that $m, n \ge 1$.
\end{remark}

\subsection{Main theorem}\label{subsec:main theorem}

Note that $L=J(2m+1,2n+1)$ is hyperbolic if and only if $m, n \not\in \{-1,0\}$. 
Since 
the double twist link $L$ is alternating, the notions of the genus and the Thurston norm 
of $L$ are equivalent (see Remark~\ref{rmk:Alexander-Thurston}). 

We are now ready to show our main theorem of this paper. 

\begin{theorem} \label{hyp_thm}
For the (hyperbolic) double twist link $L=J(2m+1,2n+1)$ with $m, n \not\in \{-1,0\}$ and with any orientation, 
the twisted Alexander polynomial $\D_{L,\rho_0}(t)$ determines the genus $g(L)$ (or equivalently the 
Thurston norm). Moreover $L$ is fibered if and only if 
$\D_{L,\rho_0}(t)$ is monic. 
\end{theorem}

Note that the character of our lift $\rho_0:G(L)\to SL(2,\C)$ has the 
form $\chi_{\rho_0} = (2,2,z_0)$ for some $z_0 \not\in \R$ satisfying $R(2,2,z_0)=0$. 

As mentioned in Remark~\ref{rmk:similar}, 
we consider only the case $m,n \ge 1$. The other cases are similar. 
To prove Theorem \ref{hyp_thm}, by Propositions \ref{pro:4.1}--\ref{pro:4.4}, 
it suffices to show that neither $Y$, $Y'$, $Y_{j,k}$, $Y'_{j,k}$, $Z_l$, $Z'_l$, $Z_{j,k}$ nor $Z'_{j,k}$ contains characters of the form $(2,2,z)$ with $z \not\in \R$. Except the cases of $Y'$ and $Z'_l$, the other ones are clear. 

We now show these two cases.

\begin{lemma} \label{Z'_l}
If $(x,y,z) \in Z'_l$ then $\disp{v=2 \cos \frac{2l\pi}{n-1}\in\R}$ and
$$z=-v^3+v^2 x y-v (x^2+ y^2-3) +x y.$$
Hence, $Z'_l$ does not contain characters of the form $(2,2,z)$ with $z \not\in \R$.
\end{lemma} 

\begin{proof}
Since $(x,y,z) \in Z'_l$, we have $S_m(z) = v S_{m-1}(z)$. 
Combining this with the equality $S^2_m(z) + S^2_{m-1}(z) - z S_m(z) S_{m-1}(z) =1$ 
in Lemma \ref{chev0}, we get $$S^2_{m-1}(v) = \frac{1}{v^2+1-zv}.$$ 
Hence
\begin{eqnarray*}
v  &=& \big( x S_m(z) - y S_{m-1}(z) \big) \big( y S_m(z) - x S_{m-1}(z) \big) \\
&&-  z \big( S^2_m(z) + S^2_{m-1}(z) \big) + 4 S_m(z) S_{m-1}(z) \\
&=& (4 v - v x^2 + x y + v^2 x y - v y^2 - z - v^2 z) S^2_{m-1}(v) \\
&=& \frac{4 v - v x^2 + x y + v^2 x y - v y^2 - z - v^2 z}{v^2-zv+1}.
\end{eqnarray*}
By solving for $z$ (in terms of $v$, $x$, $y$), we obtain the desired formula.
\end{proof}

\begin{lemma} \label{Y'}
If $(x,y,z) \in Y'$ then
$$z=(n^2+n)(x+y)^2 + xy+2.$$
Hence, $Y'$ does not contain characters of the form $(2,2,z)$ with $z \not\in \R$.
\end{lemma} 

\begin{proof}
The proof of Lemma \ref{Y'} is similar to that of Lemma \ref{Z'_l}.
\end{proof}

In the view of Theorem~\ref{hyp_thm} and its proof, 
we may propose the following problem. 

\begin{problem}\label{problem:parabolic}
For an oriented hyperbolic $2$-bridge link $L$ and its parabolic representation 
$\rho:G(L)\to SL(2,\C)$, does the twisted Alexander polynomial 
$\D_{L,\rho}(t)$ determine the genus $g(L)$ and fiberedness of $L$?
\end{problem}

If Problem~\ref{problem:parabolic} has an affirmative answer, 
then Conjecture~\ref{conj:hyp} holds true for all hyperbolic $2$-bridge links. 
However as we will see in the next section, 
a part of Problem~\ref{problem:parabolic} has a negative answer.

\section{Parabolic representations}\label{section:parabolic}

Let us recall that 
a non-abelian representation $\rho:G(L)\to SL(2,\C)$ is parabolic if 
the images of all the meridians of $L$ by $\rho$ are matrices with the trace two. 

For positive integers $m$, $n$ and $p$, consider the 2-bridge link $C(2m,2n,-2p)$ in the Conway notation 
(see Figure~3), where positive numbers $2m,2n$ correspond to right-handed twists 
and negative number $-2p$ corresponds to left-handed twists. 
It is the rational link corresponding to the continued fraction $$[2m,2n,-2p] = 2m+\frac{1}{2n - \frac{1}{2p}}$$ 
and is the 2-bridge link $\big( 2m(4np-1)+2p, 4np-1 \big)$ in the Schubert notation. 
Note that $C(2m,2n,-2p)$ is a hyperbolic link.

\begin{theorem}\label{thm:main} 
For the $2$-bridge link $L=C(2m,2n,-2p)$, with $m, n,p$ being odd positive integers and $m \not= p$, 
the twisted Alexander polynomial $\D_{L,\rho_0}(t)$ detects the genus of $L$. Moreover,
\begin{enumerate}
\item
If $\gcd(m,p) = 1$, then all parabolic representations detect $g(L)$.
\item
If $\gcd(m,p) \ge 3$, then not all parabolic representations detect $g(L)$. 
\end{enumerate}
\end{theorem}

\begin{remark}\label{rmk:knot}
We do not know if there is a parabolic representation of a $2$-bridge \textit{knot} $K$ which does not detect the 
genus of $K$ (see \cite[Theorem~1.2]{MT}). 
\end{remark}

\subsection{Chebyshev polynomials (2)}
In this subsection we prepare two lemmas for the Chebyshev polynomials. 

\begin{lemma}\label{chev6}
For integers $k,l$ we have
\begin{equation} \label{equality}
S_k(q) S_{l-1}(q) - S_{k-1}(q) S_l(q) = S_{l-k-1}(q).
\end{equation}
\end{lemma}

\begin{proof}
It suffices to show \eqref{equality} for $q \not= \pm 2$. 
When $q \not= \pm 2$, we write $q = v+v^{-1}$ for some $v \not= \pm 1$. 
By Lemma \ref{chev1}, the LHS of \eqref{equality} is equal to
\begin{eqnarray*}
&& \frac{v^{k+1} - v^{-(k+1)}}{v -  v^{-1}}
\cdot \frac{v^{l} - v^{-l}}{v -  v^{-1}} - \frac{v^{k} - v^{-k}}{v -  v^{-1}} 
\cdot \frac{v^{l+1} - v^{-(l+1)}}{v -  v^{-1}} \\
&=& \frac{-(v^{k+1-l} + v^{l-(k+1)}) + (v^{k-(l+1)} + v^{l+1-k })}{(v -  v^{-1})^2} \\
&=& \frac{v^{l-k} - v^{k-l}}{v -  v^{-1}}
\end{eqnarray*}
which is also equal to the RHS of \eqref{equality}.
\end{proof}

\begin{lemma} \label{gcd}
For integers $k,l$ we have
$$\gcd(S_{k-1}(q), S_{l-1}(q)) = S_{\gcd(k,l)-1}(q).$$
\end{lemma}

\begin{proof}
The lemma follows from the fact that 
$$\gcd(v^k-v^{-k}, v^l - v^{-l}) = v^{\gcd(k,l)} - v^{-\gcd(k,l)}$$
and Lemma \ref{chev1}.
\end{proof}

\subsection{Two bridge links $C(2m,2n,-2p)$}\label{subsec:C(2m,2n,-2p)}

In this subsection we prove Theorem~\ref{thm:main}. 
To this end we first give a presentation of the link group of $C(2m,2n,-2p)$. 

\begin{figure}\label{fig:3}
\setlength{\unitlength}{0.09mm}
\thicklines{
\begin{picture}(300,560)(385,-20)
\put(-130,250){\line(1,0){32}}
\put(-130,350){\line(1,0){33}}
\put(-97,250){\line(3,4){30}}
\put(-53,310){\line(3,4){30}}
\put(-22,250){\line(3,4){30}}
\put(24,310){\line(3,4){30}}
\put(-96,349){\line(3,-4){74}}
\put(-21,350){\line(3,-4){74}}
\put(146,250){\line(3,4){30}}
\put(190,310){\line(3,4){30}}
\put(146,350){\line(3,-4){75}}
\put(220,250){\line(3,4){30}}
\put(265,310){\line(3,4){30}}
\put(220,350){\line(3,-4){75}}
\put(65,294){{\tiny $\bullet~\bullet~\bullet$}}
\put(296,250){\line(1,0){31}}
\put(296,350){\line(1,0){502}}
\put(-130,450){\line(1,0){1345}}
\put(766,250){\line(1,0){32}}
\put(798,250){\line(3,4){74}}
\put(875,250){\line(3,4){74}}
\put(798,349){\line(3,-4){30}}
\put(844,290){\line(3,-4){30}}
\put(875,350){\line(3,-4){30}}
\put(922,290){\line(3,-4){30}}
\put(1036,250){\line(3,4){75}}
\put(1036,350){\line(3,-4){30}}
\put(1080,290){\line(3,-4){30}}
\put(1110,250){\line(3,4){75}}
\put(1110,350){\line(3,-4){30}}
\put(1158,290){\line(3,-4){30}}
\put(1188,250){\line(1,0){29}}
\put(1185,350){\line(1,0){31}}
\put(960,294){{\tiny $\bullet~\bullet~\bullet$}}
\put(-130,0){\line(0,1){250}}
\put(-130,350){\line(0,1){100}}
\put(325,100){\line(0,1){150}}
\put(325,100){\line(1,0){25}}
\put(350,0){\line(3,4){74}}
\put(425,0){\line(3,4){74}}
\put(348,99){\line(3,-4){30}}
\put(394,40){\line(3,-4){30}}
\put(425,100){\line(3,-4){30}}
\put(471,40){\line(3,-4){30}}
\put(591,0){\line(3,4){75}}
\put(591,100){\line(3,-4){30}}
\put(634,40){\line(3,-4){30}}
\put(665,0){\line(3,4){75}}
\put(665,100){\line(3,-4){30}}
\put(708,40){\line(3,-4){30}}
\put(767,100){\line(0,1){150}}
\put(738,100){\line(1,0){30}}
\put(513,44){{\tiny $\bullet~\bullet~\bullet$}}
\put(1216,0){\line(0,1){250}}
\put(1216,350){\line(0,1){100}}
\put(-131,0){\line(1,0){480}}
\put(738,0){\line(1,0){479}}
\put(450,125){{\small $2n$ crossings}}
%
\put(510,341){$\blacktriangleright$}
\put(510,441){$\blacktriangleleft$}
\put(110,-10){$\blacktriangleright$}
\put(975,-10){$\blacktriangleright$}
%
\put(1000,500){\line(0,-1){40}}
\put(1000,440){\line(0,-1){40}}
\put(985,385){$\blacktriangledown$}
\put(960,390){$a$}
\put(0,50){\line(0,-1){40}}
\put(0,-10){\line(0,-1){40}}
\put(-14,45){$\blacktriangle$}
\put(20,43){$b$}
\put(-120,360){{\small $a_0$}}
\put(-125,260){{\small $b_0$}}
\put(300,360){{\small $a_m$}}
\put(300,260){{\small $b_m$}}
\put(760,364){{\small $e_p$}}
\put(760,264){{\small $f_p$}}
\put(1178,360){{\small $e_0$}}
\put(1190,262){{\small $f_0$}}
\put(335,112){{\small $c_0$}}
\put(315,10){{\small $d_0$}}
\put(725,112){{\small $c_n$}}
\put(745,10){{\small $d_n$}}

\put(0,200){{\small $2m$ crossings}}
\put(880,200){{\small $-2p$ crossings}}
\end{picture}
}
\caption{The $2$-bridge link $L=C(2m,2n,-2p)$ with $m,n,p \ge 1$ and the generators of $G(L)$.}
\end{figure}

\begin{proposition}
The link group of $L=C(2m,2n,-2p)$ has a presentation 
$$G(L) = \la a,b \mid  aw = wa\ra$$
where $$w= (b^{-1}a)^m \left[ a^{-1} \left( (a^{-1}b)^m(ab^{-1})^m \right)^{-n} b \left( (a^{-1}b)^m(ab^{-1})^m \right)^{n} \right]^p$$ and 
$a$ and $b$ are meridians of $L$ depicted in Figure~3. 
\end{proposition}

\begin{proof}
By applying the Wirtinger algorithm to the leftmost twist region in Figure~3 
and by induction we have
$$a_m = (a_0^{-1}b_0)^m a_0 (b_0^{-1}a_0)^m, \quad b_m = (a_0^{-1}b_0)^m b _0(b_0^{-1}a_0)^m.
$$
Similarly, the middle and rightmost twist regions give
\begin{eqnarray*}
c_n &=& (d_0 c_0^{-1})^n c_0 (c_0 d_0^{-1})^n, \quad d_n = (d_0 c_0^{-1})^n d_0 (c_0 d_0^{-1})^n,\\
e_p &=& (e_0^{-1}f_0)^p e_0 (f_0^{-1} e_0)^p, \quad f_p = (e_0^{-1}f_0)^p f_0 (f_0^{-1} e_0)^p.
\end{eqnarray*}
We have $G(L) = \la a,b \mid a_m = e_p \ra$. Since $e_0 = a$ and $f_0 = d_n$, the relation $a_m = e_p$ is equivalent to $aw = wa$ where
$$w = (b^{-1}a)^m(a^{-1}d_n)^p.$$
Finally, since $d_n = (b b_m^{-1})^n b (b_m b^{-1})^n$ and $b_m = (a^{-1}b)^m b(b^{-1}a)^m$ we obtain the desired presentation of $G(L)$.
\end{proof}

Let $\rho:G(L)\to SL(2,\C)$ be a representation and 
$r=awa^{-1}w^{-1}$. We have 
$$\frac{\partial r}{\partial b} 
= a \left(\frac{\partial w}{\partial b} - wa^{-1}w^{-1} \frac{\partial w}{\partial b}\right) 
= a (1 - wa^{-1}w^{-1})\frac{\partial w}{\partial b}.$$
Hence 
$\disp{\Delta_{L,\rho} (t)= \det \Phi\left(\frac{\partial r}{\partial b}\right) 
\big/ \det \Phi(1-a) = \det \Phi\left(\frac{\partial w}{\partial b}\right)}$.

For $k \ge 0$ and $h \in G(L)$, let $\delta_k(h) = 1 + h + \cdots +h^{k}$. Let $v = (a^{-1}b)^m(ab^{-1})^m$ and $u = a^{-1} v^{-n} b v^n$. Then $w=(b^{-1}a)^m u^p $. 
 We have
\begin{eqnarray*}
\frac{\partial w}{\partial b} &=& \delta_{m-1}(b^{-1}a)(-b^{-1}) + (b^{-1}a)^m \delta_{p-1}(u) \frac{\partial u}{\partial b}, \\
\frac{\partial u}{\partial b} &=& \frac{\partial (a^{-1} v^{-n} b v^n)}{\partial b} \\
                              &=& a^{-1} \delta_{n-1}(v^{-1}) \frac{\partial v^{-1}}{\partial b} 
                                  + a^{-1} v^{-n} \left( 1 + b \delta_{n-1}(v) \frac{\partial v}{\partial b} \right)\\
                              &=& a^{-1} v^{-n} \left( 1 + (b-1) \delta_{n-1}(v) \frac{\partial v}{\partial b}\right),\\
\frac{\partial v}{\partial b} &=& \frac{\partial (a^{-1}b)^m(ab^{-1})^m}{\partial b} \\
                                   &=& \delta_{m-1}(a^{-1}b) a^{-1} +   (a^{-1}b)^m   \delta_{m-1}(ab^{-1})(-ab^{-1}).                  
\end{eqnarray*}

For $h\in G(L)$, we denote $\rho(h)$ by the capital letter $H$ for simplicity. 
With the orientation of $L=C(2m,2n,-2p)$ as in Figure~3, the genus is given by $g(L) = 1$. 
(This fact can be proved by applying Seifert's algorithm to the reduced alternating diagram 
of $L$ corresponding to the continued fraction $[2m,2n-1,1,2p-1]$.) 
Moreover, $\Phi(a) = t A$ and $\Phi(b) = t B$. Then $\Phi(v) = V$ and $\Phi(u) = U$. 

The highest degree term (in $t$) of $\Delta_{L, \rho}(t)$ is 
\begin{eqnarray*}
&& t^{0} \det \rho \Big( (b^{-1}a)^m \delta_{p-1}(u)  \times a^{-1} v^{-n} b \delta_{n-1}(v) \times  (a^{-1}b)^m   \delta_{m-1}(ab^{-1})(-ab^{-1}) \Big)\\
&=& t^0 \, \frac{T_p(\bar{u})-2}{\bar{u}-2} \cdot \frac{T_n(\bar{v})-2}{\bar{v}-2} \cdot \frac{T_m(z)-2}{z-2}
\end{eqnarray*}
by Lemma \ref{chev5}.
Similarly, the lowest one is $\disp{t^{-4} \, \frac{T_p(\bar{u})-2}{\bar{u}-2} \cdot \frac{T_n(\bar{v})-2}{\bar{v}-2} \cdot \frac{T_m(z)-2}{z-2}}$.

Consider a parabolic representation $\rho:G(L)\to SL(2,\C)$ 
given by 
$$A = \rho(a) = \begin{pmatrix}
1&1\\
0&1
\end{pmatrix} \quad \text{and} \quad 
B = \rho(b) = \begin{pmatrix}
1&0\\
2-z&1
\end{pmatrix},$$
where $z$ satisfies the matrix equation $AW = WA$. Here $W= \rho(w)$. 
Note that $z = \tr AB^{-1}$ holds.

Let $W_{ij} \in \C[z]$ denote the $(i,j)$-entry of $W$. 
Riley showed in \cite{Ri} that $W_{21} = (2-z) W'_{21}$ for some $W'_{21} \in \C[z]$ 
and that the matrix equation $AW = WA$ is equivalent to a single equation $W'_{21}=0$. 
We call $W_{21}'$ the \textit{Riley polynomial} of $L$. 

We now compute $W'_{21}$ for $L=C(2m,2n,-2p)$. 
Since $$A^{-1}B = \begin{pmatrix}
z-1 & -1\\
2-z & 1
\end{pmatrix} \qquad \text{and} \qquad AB^{-1} = \begin{pmatrix}
z-1 & 1\\
z-2 & 1
\end{pmatrix},$$ 
by Lemma \ref{chev5} we have
\begin{eqnarray*}
V  &=& (A^{-1}B)^m(AB^{-1})^m \\
&=& \begin{pmatrix}
S_m(z)-S_{m-1}(z) & -S_{m-1}(z)\\
(2-z)S_{m-1}(z) & S_m(z)-(z-1)S_{m-1}(z) 
\end{pmatrix} \\
&& \times
\begin{pmatrix}
S_m(z)-S_{m-1}(z) & S_{m-1}(z)\\
(z-2)S_{m-1}(z) & S_m(z)-(z-1)S_{m-1}(z) 
\end{pmatrix} \\
&=& \begin{pmatrix}
V_{11} & V_{12} \\
V_{21} & V_{22}
\end{pmatrix}
\end{eqnarray*}
where
\begin{eqnarray*}
V_{11} &=& S^2_m(z) - 2 S_m(z) S_{m-1}(z) + (3 - z) S^2_{m-1}(z), \\
V_{12} &=& (z-2) S^2_{m-1}(z), \\
V_{21} &=& -(z-2)^2 S^2_{m-1}(z), \\
V_{22} &=& S^2_m(z) + (2 - 2z) S_m(z) S_{m-1}(z) + (3-3z+z^2) S^2_{m-1}(z).
\end{eqnarray*}
Then
$$V^n = \begin{pmatrix}
S_n(\bar{v}) - V_{22} S_{n-1}(\bar{v}) & V_{12} S_{n-1}(\bar{v}) \\
V_{21} S_{n-1}(\bar{v}) & S_n(\bar{v}) - V_{11} S_{n-1}(\bar{v})
\end{pmatrix} = \begin{pmatrix}
\alpha & \beta \\
\gamma & \delta
\end{pmatrix}$$ 
where $\bar{v} = \tr V$.  Hence
\begin{eqnarray*}
U &=& \begin{pmatrix}
U_{11} & U_{12} \\
U_{21} & U_{22}
\end{pmatrix} = A^{-1} V^{-n} B V^n \\
&=& \begin{pmatrix}
1 & -1 \\
0 & 1
\end{pmatrix}
\begin{pmatrix}
\delta & - \beta \\
- \gamma & \alpha
\end{pmatrix} 
 \begin{pmatrix}
1 & 0 \\
2-z & 1
\end{pmatrix}  
\begin{pmatrix}
\alpha & \beta \\
\gamma & \delta
\end{pmatrix}
\\
&=& \begin{pmatrix}
1 + (z-2) \alpha (\alpha + \beta) & -1 + (z-2) \beta (\alpha + \beta) \\
-(z-2) \alpha^2 & 1 - (z-2) \alpha \beta
\end{pmatrix}.
\end{eqnarray*}
Since $W=(B^{-1}A)^m U^p$, we have
\begin{eqnarray*}
W &=& \begin{pmatrix}
S_m(z)-(z-1)S_{m-1}(z) & S_{m-1}(z)\\
(z-2)S_{m-1}(z) & S_m(z)-S_{m-1}(z) 
\end{pmatrix}
 \\
&& \times
\begin{pmatrix}
S_p(\bar{u}) - U_{22} S_{p-1}(\bar{u}) & U_{12} S_{p-1}(\bar{u}) \\
U_{21} S_{p-1}(\bar{u}) & S_p(\bar{u}) - U_{11} S_{p-1}(\bar{u})
\end{pmatrix}
\end{eqnarray*}
where $\bar{u} = \tr U$. 
This implies that
$$
W_{21} = U_{21}  \left( S_m(z)-S_{m-1}(z) \right) S_{p-1}(\bar{u}) 
+  (z-2)S_{m-1}(z) \left( S_p(\bar{u}) - U_{22} S_{p-1}(\bar{u}) \right).
$$
Since $U_{22} = 1 - (z-2) \alpha \beta$ and $U_{21} =  -(z-2) \alpha^2$ we have $$W_{21} = (2-z) W'_{21},$$ 
where the Riley polynomial of $L$ is given by 
\begin{eqnarray*}
W'_{21} &=& \alpha^2  \big( S_m(z)-S_{m-1}(z) \big) S_{p-1}(\bar{u}) \\
&& - \,  S_{m-1}(z) \Big( S_p(\bar{u}) + ( (z-2) \alpha \beta - 1) S_{p-1}(\bar{u}) \Big).
\end{eqnarray*}

Since the holonomy representation $\rho_0:G(L)\to SL(2,\C)$ 
is one of the parabolic representations, it has the form 
$$\rho(a) = \begin{pmatrix}
1&1\\
0&1
\end{pmatrix} \quad \text{and} \quad 
\rho(b) = \begin{pmatrix}
1&0\\
2-z_0&1
\end{pmatrix}$$
for some $z_0 \not\in \R$ satisfying $W'_{21}(z_0)=0$. 

To prove Theorem \ref{thm:main}, we determine all complex numbers $z$ satisfying both $W'_{21}(z)=0$ and 
$$\frac{T_p(\bar{u})-2}{\bar{u}-2} \cdot \frac{T_n(\bar{v})-2}{\bar{v}-2} \cdot \frac{T_m(z)-2}{z-2}=0.$$
We consider the following three cases.

\textbf{Case 1.} Suppose $\disp{\frac{T_m(z)-2}{z-2} = 0}$. Since $m$ is odd, by Lemma \ref{chev3} we have 
$\disp{z = 2 \cos \frac{2j\pi}{m}}$ for some $\disp{1 \le j \le \frac{m-1}{2}}$. 
Then, by Lemma \ref{chev1}, 
$$S_{m-1}(z)=\frac{\sin m \frac{2j\pi}{m}}{\sin \frac{2j\pi}{m}}=0, 
\quad \text{and} \quad S_m(z)=\frac{\sin (m+1)\frac{2j\pi}{m}}{\sin \frac{2j\pi}{m}}=1.$$
It is easy to see that $V = I$ and $$U = A^{-1} V^{-n} B V^n =A^{-1}B.$$ 
This implies that $W'_{21} = S_{p-1}(z)$. 

By Lemma \ref{gcd} we have $S_{\gcd(m,p)-1}(z) = \gcd(S_{m-1}(z), S_{p-1}(z))$. 

If $\gcd(m,p) = 1$, then $\gcd(S_{m-1}(z), S_{p-1}(z)) =1$. Since $S_{m-1}(z)  = 0$, we have
$W'_{21} = S_{p-1}(z) \not= 0$. 

If $\gcd(m,p) \ge 3$, we can choose $z'\in\R$ such that both $\disp{\frac{T_m(z')-2}{z'-2}}$ and $S_{p-1}(z')$ are zero. 
Indeed, choose some $\disp{1 \le j \le \frac{1}{2} \gcd(m,p)}$ and take $\disp{z'=2\cos \frac{2j \pi}{\gcd(m,p)}}$. 
Then the parabolic representation $\rho'$ corresponding to the root $z'$ of $W'_{21}$ does not detect 
the genus of $L$.

\textbf{Case 2.} Suppose $\disp{\frac{T_n(\bar{v})-2}{\bar{v}-2} = 0}$. Since $n$ is odd, 
we have $S_{n-1}(\bar{v})  = 0$ and $S_{n}(\bar{v})  = 1$. 
Then $$V^n = \begin{pmatrix}
S_n(\bar{v}) - V_{22} S_{n-1}(\bar{v}) & V_{12} S_{n-1}(\bar{v}) \\
V_{21} S_{n-1}(\bar{v}) & S_n(\bar{v}) - V_{11} S_{n-1}(\bar{v})
\end{pmatrix} = I$$
and $U=A^{-1} V^{-n} B V^n=A^{-1}B$. 
Hence, by Lemma \ref{chev1} we have $$W'_{21} = \big( S_m(z)-S_{m-1}(z) \big) S_{p-1}(z) - S_{m-1}(z) \big( S_p(z) -  S_{p-1}(z) \big) = S_{p-m-1}(z).$$

If $S_{p-m-1}(z) = 0$, then since $m \not= p$ we have $z \in \R$ (strictly between $-2$ and $2$) by Lemma \ref{chev2}. 
By a direct calculation and Lemma \ref{chev0} we have
\begin{eqnarray*}
\bar{v} = \tr V &=& 2 \Big( S^2_m(z) + S^2_{m-1}(z) - z S_m(z) S_{m-1}(z)\Big) + (z-2)^2 S^2_{m-1}(z) \\
&=& 2 + (z-2)^2 S^2_{m-1}(z)
\end{eqnarray*}
which is a real number greater than or equal to $2$. This contradicts $S_{n-1}(\bar{v})  = 0$.

\textbf{Case 3.}
Suppose $\disp{\frac{T_p(\bar{u})-2}{\bar{u}-2} = 0}$. Since $p$ is odd, we have $S_{p-1}(\bar{u}) = 0$ and $S_{p}(\bar{u}) = 1$. 
Then $W'_{21}  = -S_{m-1}(z)$. 

If $S_{m-1}(z) = 0$, then $S^2_m(z) = 1$ by Lemma \ref{chev0}. It is easy to see that
$V  = I$ and $$U = A^{-1} V^{-n} B V^n =A^{-1}B.$$ 
This implies that $\bar{u} = z$.
Hence $S_{p-1}(z) = S_{p-1}(\bar{u}) = 0$. 
It cannot occur that both $S_{m-1}(z) = 0$ and $S_{p-1}(z) = 0$ if $\gcd(m,p)=1$. 

From the above discussion we have proved Theorem \ref{thm:main} for the orientation in Figure~3. 

We now consider another orientation of $L$ which is obtained from the one in Figure~3 by changing the orientation of the component  corresponding to the meridian $a$. With this orientation we have $g(L) = m + p -1$. Moreover $\Phi(a) = t^{-1} A$ and $\Phi(b) = t B$. Then $\Phi(v) = V$ and $\Phi(u) = t^2 U$. 

In this case the highest degree term (in $t$) of $\Delta_{L, \rho}(t)$ is 
\begin{eqnarray*}
&& t^{4p-2} \det \left( (b^{-1}a)^m u^{p-1} \times a^{-1} v^{-n} b \delta_{n-1}(v) \times \delta_{m-1}(a^{-1}b) a^{-1} \right) \\
&=& t^{4p-2} \, \frac{T_{n-1}(\bar{v})-2}{\bar{v}-2}
\end{eqnarray*}
and the lowest one is $\disp{t^{2-4m} \, \frac{T_{n-1}(\bar{v})-2}{\bar{v}-2}}$. 
A similar argument as in Case 2 above shows that all parabolic representations detect the genus of $L$.

This completes the proof of Theorem~\ref{thm:main}.


\end{document}